\documentclass[a4paper,11pt,reqno]{article}

\usepackage[utf8]{inputenc}
\usepackage[T1]{fontenc}
\usepackage{lmodern}
\usepackage[english]{babel}
\usepackage{microtype}

\usepackage{amsmath,amssymb,amsfonts,amsthm}
\usepackage{mathtools,accents}
\usepackage{mathrsfs}
\usepackage{aliascnt}
\usepackage{braket}
\usepackage{bm}

\usepackage[a4paper,margin=3cm]{geometry}
\usepackage[citecolor=blue,colorlinks]{hyperref}

\usepackage{enumerate}
\usepackage{xcolor}

\makeatletter
\g@addto@macro\@floatboxreset\centering
\makeatother

\makeatletter
\def\newaliasedtheorem#1[#2]#3{
  \newaliascnt{#1@alt}{#2}
  \newtheorem{#1}[#1@alt]{#3}
  \expandafter\newcommand\csname #1@altname\endcsname{#3}
}
\makeatother

\numberwithin{equation}{section}

\newtheoremstyle{slanted}{\topsep}{\topsep}{\slshape}{}{\bfseries}{.}{.5em}{}

\theoremstyle{plain}
\newtheorem{theorem}{Theorem}[section]
\newaliasedtheorem{proposition}[theorem]{Proposition}
\newaliasedtheorem{lemma}[theorem]{Lemma}
\newaliasedtheorem{corollary}[theorem]{Corollary}
\newaliasedtheorem{counterexample}[theorem]{Counterexample}

\theoremstyle{definition}
\newaliasedtheorem{definition}[theorem]{Definition}
\newaliasedtheorem{question}[theorem]{Question}
\newaliasedtheorem{example}[theorem]{Example}

\theoremstyle{remark}
\newaliasedtheorem{remark}[theorem]{Remark}

\let\altphi\phi
\let\phi\varphi
\let\varphi\altphi
\let\altphi\undefined

\newcommand{\di}{\mathop{}\!\mathrm{d}}

\newcommand{\res}{\mathop{\hbox{\vrule height 7pt width .5pt depth 0pt
\vrule height .5pt width 6pt depth 0pt}}\nolimits}

\DeclareMathOperator{\supp}{supp}

\newcommand{\Ch}{{\sf Ch}}

\DeclareMathOperator{\Lip}{Lip}

\newcommand{\dist}{\mathsf{d}}

\newcommand{\meas}{\mathfrak{m}}

\DeclareMathOperator{\RCD}{RCD}
\DeclareMathOperator{\BE}{BE}
\newfont{\tmpf}{cmsy10 scaled 2500}

\begin{document}
\title{Bakry-\'Emery conditions on almost smooth metric measure spaces}
\author{Shouhei Honda
\thanks{Tohoku University, \url{shonda@m.tohoku.ac.jp}}} 
\maketitle

\begin{abstract}In this short note, we give a sufficient condition for almost smooth compact metric measure spaces to satisfy the Bakry-\'Emery condition $\BE (K, N)$. 
The sufficient condition is satisfied for the glued space of any two (not necessary same dimensional) closed pointed Riemannian manifolds at their base points. This tells us that the $\BE$ condition is strictly weaker than the $\RCD$ condition even in this setting, and that the local dimension is not constant even if the space satisfies the $\BE$ condition with the coincidence between the induced distance by the Cheeger energy and the original distance. 
In particular, the glued space gives a first example with a Ricci bound from below in the Bakry-\'Emery sense, whose local dimension is not constant. We also give a necessary and sufficient condition for such spaces to be $\RCD(K, N)$ spaces.
\end{abstract}

\tableofcontents

\section{Introduction}

Let $(X, \dist, \meas)$ be a compact metric measure space, that is, $(X, \dist)$ is a compact metric space with $\supp \meas =X$ and $\meas (X)<\infty$. 
There are several definitions of `lower Ricci bounds on $(X, \dist, \meas)$', whose studies are very quickly, widely developed now.
We refer to, \cite{LottVillani} by Lott-Villani, \cite{Sturm06} by Sturm, and \cite{AmbrosioGigliSavare14} by Ambrosio-Gigli-Savar\'e, as their pioneer works.

In this paper, we forcus on two of them. One of them is the \textit{Bakry-\'Emery $\mathrm{(}\BE\mathrm{)}$ condition}  \cite{BE} by Bakry-\'Emery, denoted by $\BE(K, N)$,  the other is the \textit{Riemannian curvature dimension $\mathrm{(}\RCD\mathrm{)}$ condition} \cite{AmbrosioGigliSavare14} by Ambrosio-Gigli-Savar\'e (in the case when $N=\infty$), \cite{Gigli0} by Gigli (in the case when $N<\infty$), denoted by $\RCD(K, N)$.
Both notions give us meanings that the Ricci curvature of $(X, \dist, \meas)$ is bounded below by $K$, and the dimension of $(X, \dist, \meas)$ is bounded above by $N$ in synthetic sense.
The $\BE(K, N)$ condition is roughly stated by:
\begin{equation}\label{0}
\frac{1}{2}\Delta |\nabla f|^2 \ge \frac{(\Delta f)^2}{N}+\langle \nabla \Delta f, \nabla f\rangle +K|\nabla f|^2
\end{equation}
holds in a weak form for all `nice' functions $f$ on $X$ (Definition \ref{bedef}).
It is known that if $(X, \dist)$ is a $n$-dimensional smooth Riemannian manifold $(M^n, g)$ and $\meas$ is the Riemannian (or equivalently, the Hausdorff) measure, then, the $\BE(K, N)$ condition (\ref{0}) is equivalent to satisfying $n\le N$ and $\mathrm{Ric}_{M^n}^g \ge K$, and that the $\BE(K, N)$ condition is also equivalent to some gradient estimates on the heat flow, so called Bakry-\'Emery/Bakry-Ledoux gradient estimates.

In general, the implication from $\RCD(K, N)$ to $\BE(K, N)$ is always satisfied. The converse is true under adding a some property, so-called the `\textit{Sobolev to Lipschitz property}', which is introduced in \cite{Gigli0} (Definition \ref{def:RCDspaces}).
This property (with $\BE$), in a point of view in geometric analysis, plays a role to get the coincidence between the analytic distance $\dist_{\Ch}$ (that is, the induced distance by the Cheeger energy) and the original (geometric) distance $\dist$. 
Moreover, the $\RCD$ condition also implies the Sobolev to Lipschitz property. Thus, the following equivalence is known:
\begin{equation}\label{444444}
\RCD(K, N) \Longleftrightarrow \BE(K, N) + `\mathrm{Sobolev\,to\, Lipschitz\,property}'.
\end{equation}

The RHS of (\ref{444444}) is also called the metric $\BE(K, N)$ condition.
Thus, to keep the short presentation, we adopt the RHS of (\ref{444444}) as the definition of $\RCD(K, N)$ condition in this paper (Definition \ref{def:RCDspaces}).
We refer to, \cite{AmbrosioGigliSavare15} by Ambrosio-Gigli-Savar\'e, \cite{AmbrosioMondinoSavare}, \cite{AmbrosioMondinoSavare16} by Ambrosio-Mondino-Savar\'e, and \cite{ErbarKuwadaSturm} by Erbar-Kuwada-Sturm for the details.

In these observation, more precisely, `$\RCD$' should be replaced by `$\RCD^*$'. However, since the equivalence between $\RCD$ and $\RCD^*$ spaces is also recently established in \cite{CavMil} by Cavalletti-Milman, we use the notation `$\RCD$' only for simplicity.

In this paper, we discuss the condition:
\begin{equation}\label{22}
\BE(K, N) + `\dist_{\Ch}=\dist'.
\end{equation}

One of the goals in this paper is to provide an example satisfying (\ref{22}), but it is not an $\RCD$ space.
More precisely, for any two (not necessary same dimensional) closed pointed Riemannian manifolds $(M_i^{m_i}, g_i,  p_i) (m_i \ge 2)$, the glued metric space $M_1^{m_1} * M_2^{m_2}$ at their base points with the standard measure is a $\BE (K, \max \{m_1, m_2\})$ space, where $K:= \min \{ \inf \mathrm{Ric}_{M_1^{m_1}}^{g_1}, \inf \mathrm{Ric}_{M_2^{m_2}}^{g_2}\}$ (Example \ref{1091}). It is easy to check that this metric measure space does not satisfy the Sobolev to Lipschitz property, thus, it is not a $\RCD(L, \infty)$ space for any $L \in \mathbb{R}$. 

This tells us that (\ref{22}) does \textit{not} imply the expected Bishop-Gromov inequality (Remark \ref{yy6}), and that (\ref{22}) does \textit{not} imply the constancy of the local dimension.
In particular, the glued space gives a first example with a Ricci bound from below in the Bakry-\'Emery sense, whose local dimension is not constant.
We point out a very recent result in \cite{BS18} by Bru\`e-Semola, which states that for any $\RCD(K, N)$ space, there exists a unique $k$ such that the $k$-dimensional regular set $\mathcal{R}_k$ has positive measure.
This generalizes a result of Colding-Naber in \cite{CN} for Ricci limit spaces to $\RCD$ spaces.
Thus, we know that the Sobolev-Lipschitz property is crucial to get the constant dimensional property.
Note that in \cite{KR}, Ketterer-Rajala constructed a metric measure space with the measure contraction property (MCP), which also characterize `Ricci bounds from below' in a synthetic sense, but the local dimension is not constant. Therefore, in general, MCP and BE spaces are very diferent from $\RCD$ spaces.

Moreover, we should pay attention to a similar sufficient condition in \cite{AmbrosioMondinoSavare16} by Ambrosio-Mondino-Savar\'e, so-called the local to global property, which states in our compact setting; if $(X, \dist)$ is a geodesic (or equivalently, length) space and there exists an open covering $\{U_i\}_{i \in I}$ of $X$ such that $U_i \neq \emptyset$ and that $(\overline{U_i}, \dist, \meas_{\overline{U_i}})$ satisfies the metric $\BE(K, N)$ condition, then, $(X, \dist, \meas)$ satisfies the metric $\BE(K, N)$ condition. In fact, the glued example shows that the openness of $U_i$ is essential because although $U_i:=M_i^{m_i}$ in $M_1^{m_1}*M_2^{m_2}$ satisfies the assumptions except for their openness properties, but the glued space does not satisfy the metric $\BE(K, N)$ condition for all $K, N$.

In order to justify these, we study \textit{almost smooth compact metric measure spaces}. See Definition \ref{def:asmm} for the definition, which allows us such spaces to have at least the codimension $2$ singularities. Thus, compact (Riemannian) orbifolds with the Hausdorff measure are typical examples of them. Then, the main result in this paper is roughly stated as follows; if an almost smooth compact metric measure space satisfies the $L^2$-strong compactness condition and satisfies the gradient estimates on the eigenfunctions, then, a lower bound of the Ricci tensor of the smooth part implies a $\BE$ condition (Theorem \ref{be}).
By using this, we can give a necessary and sufficient condition for such a space to be a $\RCD$ space (Corollary \ref{corrcd}).

The organization of the paper is as follows.

In section $2$, to keep the short presentation, we give a very quick introduction to calculus on metric measure spaces.

In section $3$, we study our main targets, almost smooth metric measure spaces, and prove the main results.

\textbf{Acknowledgement.}
A part of the work is done during the author's stay in Yau Mathematical Sciences Center (YMSC) at Tsinghua University.
The author would like to express his appreciation to Guoyi Xu for his warm hospitality.
He is also grateful to YMSC for giving him nice environment.
He thanks Luigi Ambrosio, Nicola Gigli, Bangxian Han and Aaron Naber for helpful comments.
Moreover, He thanks the referee for the careful reading of the manuscript and for the suggestions in the revision. 
Finally, he acknowledges the supports of the Grantin-Aid
for Young Scientists (B), 16K17585, and of the
Grant-in-Aid for Scientific Research (B), 18H01118.
\section{$\BE$ and $\RCD$ spaces}

We use the notation $B_r(x)$ for open balls and $\overline{B}_r(x)$ for $\{y:\ \dist(x,y)\leq r\}$.
We also use the standard notation $\mathrm{LIP}(X,\dist)$, $\mathrm{LIP}_c(X,\dist)$ for the
spaces of Lipschitz, compactly supported Lipschitz functions, respectively. 

Let us now recall basic facts about Sobolev spaces in metric measure spaces 
$(X,\dist,\meas)$, see \cite{AmbrosioGigliSavare13}, \cite{Gigli1} and \cite{Gigli} for a more systematic treatment of this topic. 
We shall always assume that 
\begin{itemize}
\item the metric space $(X,\dist)$ is compact with $\supp \meas =X$ and $\meas (X)<\infty$
\end{itemize}
for simplicity. 

The Cheeger energy
$\Ch=\Ch_{\dist,\meas}:L^2(X,\meas)\to [0,+\infty]$ is a convex and $L^2(X,\meas)$-lower semicontinuous functional defined as follows:
\begin{equation}\label{eq:defchp}
\Ch(f):=\inf\left\{\liminf_{n\to\infty}\frac 12\int_X(\mathrm{Lip}  f_n)^2\di\meas:\ \text{$f_n\in\Lip (X,\dist)$, $\|f_n-f\|_{L^2}\to 0$}\right\}, 
\end{equation}
where $\mathrm{Lip} f$ is the so-called slope, or local Lipschitz constant.

The Sobolev space $H^{1,2}(X,\dist,\meas)$ then concides with $\{f:\ \Ch(f)<+\infty\}$. When endowed with the norm
$$
\|f\|_{H^{1,2}}:=\left(\|f\|_{L^2(X,\meas)}^2+2\Ch(f)\right)^{1/2}
$$
this space is Banach, reflexive if $(X,\dist)$ is doubling (see \cite{AmbrosioColomboDiMarino}), and  
separable Hilbert if $\Ch$ is a quadratic form (see \cite{AmbrosioGigliSavare14}). 
According to the terminology introduced in \cite{Gigli1}, we say that a 
metric measure space $(X,\dist,\meas)$ is infinitesimally Hilbertian if $\Ch$ is a quadratic form.
  
By looking at minimal relaxed slopes and by a polarization procedure, one can then define a {\it carr\'e du champ}
$$
\Gamma:H^{1,2}(X,\dist,\meas)\times H^{1,2}(X,\dist,\meas)\rightarrow L^1(X,\meas)
$$
playing in this abstract theory the role of the scalar product between gradients (more precisely,
the duality between differentials and gradients, see \cite{Gigli1}). In infinitesimally Hilbertian metric measure
spaces, the $\Gamma$ operator
satisfies all natural symmetry, bilinearity, locality and chain rule properties, and provides integral representation to
$\Ch$: $2\Ch(f)=\int_X \Gamma(f,f)\,\dist\meas$ for all $f\in H^{1,2}(X,\dist,\meas)$. 

We can now define a densely
defined operator $\Delta:D(\Delta)\to L^2(X,\meas)$ whose domain consists of all functions $f\in H^{1,2}(X,\dist,\meas)$
satisfying
$$
 \int_X hg\dist\meas=-\int_X \Gamma(f,h)\dist\meas\quad\qquad\forall h\in H^{1,2}(X,\dist,\meas)
$$
for some $g\in L^2(X,\meas)$. The unique $g$ with this property is then denoted by $\Delta f$ (see \cite{AmbrosioGigliSavare13}).

From the point of view of Riemannian geometry, we will also adopt the following notaion instead of $\Gamma$;
$$
\langle \nabla f, \nabla g \rangle :=\Gamma (f, g), \, \quad |\nabla f|^2:=\Gamma (f, f). 
$$

We are now in a position to introduce the $\BE(K, N)$ condition (see \cite{AmbrosioMondinoSavare}, \cite{AmbrosioMondinoSavare16} and \cite{ErbarKuwadaSturm}):
\begin{definition}[$\BE$ spaces]\label{bedef}
Let $(X, \dist, \meas)$ be a compact metric measure space, let $K \in \mathbb{R}$ and let $N \in [1, \infty]$.
We say that $(X, \dist, \meas)$ is a \textit{$\BE (K, N)$ space} if 
for all $f\in D(\Delta)$ with $\Delta f\in H^{1,2}(X,\dist,\meas)$, 
Bochner's inequality
$$
\frac 12\Delta |\nabla f|^2 \geq \frac{(\Delta f)^2}{N} + \langle \nabla f,\nabla \Delta f\rangle + K|\nabla f|^2 
$$
holds in the weak form, that is, 
\begin{equation}\label{eq:boch}
\frac 12\int_X |\nabla f|^2\Delta\phi\dist\meas\geq
\int_X\phi\left(\frac{(\Delta f)^2}{N}+ \langle \nabla f,\nabla \Delta f\rangle + K|\nabla f|^2\right)\dist\meas 
\end{equation}
for all $\phi\in D(\Delta) \cap L^{\infty}(X, \meas)$ with $\phi\geq 0$ and $\Delta\phi\in L^\infty(X,\meas)$.
\end{definition}

In order to introduce the class of $\RCD(K,N)$ metric measure spaces, we follow the
$\Gamma$-calculus point of view, based on Bochner's inequality, because this is the point of view more relevant
in our proofs. However, the equivalence with the Lagrangian point
of view, based on the theory of optimal transport first proved in \cite{AmbrosioGigliSavare15} (in the case $N=\infty$) and then in
\cite{ErbarKuwadaSturm}, \cite{AmbrosioMondinoSavare} (in the case $N<\infty$).
Moreover, the following definition should be written as $\RCD^*(K, N)$ spaces. However, since it is known by \cite{CavMil} that these are equivalent notions, we use the notation $\RCD(K, N)$ only for simplicity.

\begin{definition} [$\RCD$ spaces]\label{def:RCDspaces} Let $(X,\dist,\meas)$ be a compact metric measure space, let $K \in \mathbb{R}$ and let $N \in [1, \infty]$.
We say that $(X, \dist, \meas)$ is a \textit{$\RCD(K, N)$ space} if it is a $\BE (K, N)$ space with the Sobolev-Lipschitz property, that is, 
\begin{itemize}
\item{(Sobolev to Lipschitz property)} any  
$f\in H^{1,2}(X,\dist,\meas)$ with $|\nabla f| \leq 1$ $\meas$-a.e. in $X$ 
has a $1$-Lipschitz representative.
\end{itemize}
\end{definition}

We end this section by giving the definition of local Sobolev spaces:
\begin{definition}[Sobolev spaces $H^{1,2}_0$]\label{def:loc sob}
Let $U$ be an open subset of $X$.
We denote by $H^{1,2}_0(U, \dist, \meas )$ the 
$H^{1,2}$-closure of $\mathrm{LIP}_c(U, \dist)$.
\end{definition}

In the next section, the local Sobolev spaces will play a role to localize global Sobolev functions to smooth parts via the zero capacity condition.
\section{Almost smooth metric measure space}
Let us fix a compact metric measure space $(X, \dist, \meas)$.
\subsection{Constant dimensional case}
\begin{definition}[$n$-dimensional almost smooth compact metric measure space]\label{def:asmm}
Let $n \in \mathbb{N}$. We say that $(X, \dist, \meas)$ is an \textit{$n$-dimensional almost smooth compact metric measure space associated with an open subset $\Omega$ of $X$} if the following three conditions are satisfied;
\begin{enumerate}
\item{(Smoothness of $\Omega$)} there exist an $n$-dimensional (possibly incomplete) Riemannian manifold $(M^n, g)$ and a map $\phi: \Omega \to M^n$ such that $\phi$ is a local isometry between $(\Omega, \dist)$ and $(M^n, \dist_g)$, that is, for all $p \in \Omega$ there exists an open neighborhood $U \subset \Omega$ of $p$ such that $\phi|_U$ is an isometry from $U$ to $\phi(U)$ as metric spaces;
\item{(Hausdorff measure condition)} The restricition $\meas \res_{\Omega}$ of $\meas$ to $\Omega$ coincides with the $n$-dimensional Hausdorff measure $\mathcal{H}^n$ on $\Omega$, that is, $\meas (A)=\mathcal{H}^n(A)$ holds for all Borel subset $A$ of $\Omega$;
\item{(Zero capacity condition)} $X \setminus \Omega$ has zero capacity in the following sense, that is, $\meas (X \setminus \Omega)=0$ is satisfied, there exists a sequence $\phi_i \in C^{\infty}_c(\Omega)$ such that the following two conditions hold;
\begin{enumerate}
\item for any compact subset $A \subset \Omega$, $\phi_i|_A \equiv 1$ holds for all sufficiently large $i$; 
\item it holds that $0 \le \phi_i \le 1$ and that
\begin{equation}\label{uuh}
\sup_i\int_{\Omega} |\Delta \phi_i | \dist \mathcal{H}^n<\infty.
\end{equation}
\end{enumerate}
\end{enumerate}
\end{definition}
The zero capacity condition is a kind of that `$H^{1, 2}$-capacity of $X \setminus \Omega$ is zero' whose
 standard definition is given by replacing (\ref{uuh}) by
\begin{equation}\label{ffff}
\int_{\Omega}|\nabla \phi_i|^2\dist \mathcal{H}^n \to 0 \quad (i \to \infty).
\end{equation}
See \cite{JM}. In particular, (\ref{ffff}) is satisfied if $\|\Delta \phi_i\|_{L^1} \to 0$. Compare with (2) of Proposition \ref{fun}.
\begin{remark}\label{3332}
Whenever we discuss `analysis/geometry on $\Omega$ locally', we can identify $(\Omega, \dist)$ with the smooth Riemannian manifold $(M^n, g)$ (thus, sometimes, we will use the notations $(\Omega, g), \mathrm{Ric}_{\Omega}^g$ and so on).
Note that for all $p \in M^n$ and all sufficiently small $r>0$, $B_r^g(p)$ is convex and it has a uniform lower bound on Ricci curvature.
In particular, the volume doubling condition and the Poincar\'e inequality hold locally.
Thus, Cheeger's theory \cite{Cheeger} can be applied locally.  
In particular, the \textit{Lipschitz-Lusin property} holds for all $f \in H^{1, 2}(X, \dist, \meas)$ (this notion is equivalent to that of  differentiability of functions introduced in \cite{Honda4}), that is, for all $\epsilon>0$, there exists a Borel subset $A $ of $\Omega$ such that $\meas (\Omega \setminus A)<\epsilon$ and that $f|_A$ is Lipschitz.
Combining this with the locality property of the slope on both theories in \cite{AmbrosioGigliSavare13}, in \cite{Cheeger}, yields
\begin{equation}\label{ppk}
|\nabla f|(x) = |\nabla^g (f \circ \phi^{-1})|(\phi(x)) \quad \mathcal{H}^n-a.e. x \in \Omega,
\end{equation}
where the RHS means the minimal weak upper gradient in \cite{Cheeger}.

Let us give a quick proof of (\ref{ppk}) for reader's convenience. By the Lipschitz-Lusin property with the localities of slopes as mentioned above, it suffices to check that under assuming $f \in \mathrm{LIP}(X, \dist)$, the LHS of (\ref{ppk}) is equal to $\mathrm{Lip} f$ for $\meas$-a.e. $x \in X$. Moreover, since it follows from \cite{AmbrosioGigliSavare13} that $|\nabla f|(x) \le \mathrm{Lip} f(x)$ $\meas$-a.e. $x \in X$, let us check the converse inequality.

Let $x \in \Omega$ and fix any sufficiently small $r>0$ as above.
Note that by \cite{Cheeger}, if $f_i \in \mathrm{LIP}(B_r(x), \dist)$ $L^2$-strongly converge to $f$ on $B_r(x)$, then
\begin{equation}\label{eerf}
\liminf_{i \to \infty}\int_{B_r(x)}(\mathrm{Lip} f_i)^2\dist \mathcal{H}^n \ge \int_{B_r(x)}(\mathrm{Lip} f)^2\dist \mathcal{H}^n.
\end{equation}
On the other hand, by \cite{AmbrosioGigliSavare14}, there exists a sequence $F_i \in \mathrm{LIP}(X, \dist)$ such that $F_i, \mathrm{Lip} F_i \to f, |\nabla f|$ in $L^2(X, \meas)$, respectively. Applying (\ref{eerf}) for $f_i=F_i$ shows
$$
\int_{B_r(x)}|\nabla f|^2\dist \mathcal{H}^n\ge \int_{B_r(x)}(\mathrm{Lip} f)^2\dist \mathcal{H}^n.
$$
Since $r$ is arbitrary, we have the converse inequality, $|\nabla f|(x) \ge \mathrm{Lip} f(x)$ $\meas$-a.e. $x \in X$, which completes the proof.
 
Similarly, the Sobolev space $H^{1, 2}_0(M^n, g, \mathcal{H}^n)$, which is defined by the standard way in Riemannian geometry (that is, the $H^{1, 2}$-closure of $C^{\infty}_c(M^n)$), coincides with $H^{1, 2}_0(\Omega, \dist, \meas)$. We will immediately use these compatibilities below.
\end{remark}
From now on, we use the same notation as in Definition \ref{def:asmm} (e.g. $\Omega, \phi_i$) without any attention.
\begin{proposition}\label{fun}
Let $(X, \dist, \meas)$ be an $n$-dimensional almost smooth compact metric measure space.
Then 
\begin{enumerate}
\item $\phi_i \to 1$ in $L^1(X, \meas)$ with $\sup_i \|\phi_i\|_{H^{1, 2}}<\infty$;
\item the canonical inclusion map $\iota: H^{1, 2}_0(\Omega, \dist, \mathcal{H}^n) \hookrightarrow H^{1, 2}(X, \dist, \mathcal{H}^n)$ is an isometry. In particular $(X, \dist, \mathcal{H}^n)$ is infinitesimally Hilbertian.
\end{enumerate}
\end{proposition}
\begin{proof}
Since $\phi_i(x) \to 1$ $\meas$-a.e. $x \in X$, applying the dominated convergence theorem shows that $\phi_i \to 1$ in $L^2(X, \meas)$.
Moreover, since 
$$
\int_{\Omega}|\nabla \phi_i|^2\dist \mathcal{H}^n =-\int_{\Omega}\phi_i\Delta \phi_i\dist \mathcal{H}^n \le \int_{\Omega}|\Delta \phi_i|\dist \mathcal{H}^n,
$$
we have (1).

Next, let us check (2). It is trivial that the map $\iota$ preserves the distances (we identify $H^{1, 2}_0(\Omega, \dist, \meas)$ with the image by $\iota$ for simplicity).  As written in Remark \ref{3332}, it also follows from the smoothness of $\Omega$ that $H^{1, 2}_0(\Omega, \dist, \meas)$ is a Hilbert space, and that $\phi_i f \in H^{1, 2}_0(\Omega, \dist, \meas)$ for all $f \in \mathrm{LIP}(X, \dist)$.

Fix $f \in \mathrm{LIP}(X, \dist)$. Then, since
$$
\int_X|\nabla (\phi_i f)|^2\dist \meas \le \int_X\left(2|\nabla f|^2 + 2|f|^2 |\nabla \phi_i|^2\right)\dist \meas,
$$
we have $\sup_i\|\phi_if\|_{H^{1, 2}}<\infty$.
Therefore, since $\phi_if \to f$ in $L^2(X, \meas)$, Mazur's lemma yields $f \in H^{1, 2}_0(\Omega, \dist, \meas)$.

In particular, 
\begin{equation}\label{rrf}
\Ch(\phi+\psi)+\Ch(\phi-\psi)=2\Ch(\phi )+2\Ch(\psi ) \quad \forall \phi, \psi \in \mathrm{LIP}(X, \dist).
\end{equation}

By \cite{AmbrosioGigliSavare13}, for all $F, G \in H^{1, 2}(X, \dist, \meas)$, there exist sequences $F_i, G_i \in \mathrm{LIP}(X, \dist)$ such that $F_i, G_i \to F, G$ in $L^2(X, \meas)$, respectively and that $\Ch(F_i), \Ch (G_i) \to \Ch(F), \Ch(G)$, respectively. Then, letting $i \to \infty$ in the equality (\ref{rrf}) for $\phi=F_i, \psi=G_i$ with the lower semicontinuity of the Cheeger energy shows
\begin{equation}\label{ss}
\Ch(F+G)+\Ch (F-G) \le 2\Ch(F)+2\Ch(G).
\end{equation}

Replacing $F, G$ by $F+G, F-G$, respectively yields the converse inequality, that is, we have the equality in (\ref{ss}) for $\phi=F, \psi=G$, which  proves that $H^{1, 2}(X, \dist, \meas)$ is a Hilbert space.
Thus, by \cite{AmbrosioGigliSavare14}, $\mathrm{LIP}(X, \dist)$ is dense in $H^{1, 2}(X, \dist, \meas)$.
Since we already proved that $\mathrm{LIP}(X, \dist) \subset H^{1, 2}_0(\Omega, \dist, \meas)$, we conclude.
\end{proof}
\begin{remark}
Recall that if $u_i$ $L^2$-weakly converge to $u$ in $L^{2}(X, \meas)$ with $\sup_i\|u_i\|_{H^{1, 2}}<\infty$, then, we see that $u \in H^{1, 2}(X, \dist, \meas)$ and that $\nabla u_i$ $L^2$-weakly converge to $\nabla u$. Although this statement was already proved in general setting (e.g. \cite{AmbrosioStraTrevisan} and \cite{Gigli}. See also \cite{AmbrosioHonda} and \cite{Honda2}), for reader's convenience, let us give a proof as follows.

Mazur's lemma yields the first statement, $u \in H^{1, 2}(X, \dist, \meas)$.
To get the second one, since $\sup_i\|\nabla u_i\|_{L^2}<\infty$, it is enough to check that
\begin{equation}\label{eer}
\int_X\langle \nabla u_i, f\nabla h\rangle \dist \meas \to \int_X\langle \nabla u, f\nabla h\rangle \dist \meas \quad (i \to \infty) \quad \forall f, h \in C^{\infty}_c(\Omega).
\end{equation} 
Then,
\begin{align*}
\int_X\langle \nabla u_i, f\nabla h\rangle \dist \meas &=\int_Xu_i (-\langle \nabla f, \nabla h\rangle -f\Delta h)\dist \meas  \\
&\to \int_Xu (-\langle \nabla f, \nabla h\rangle -f\Delta h)\dist \meas =\int_X\langle \nabla u, f\nabla h\rangle \dist \meas
\end{align*}
which proves (\ref{eer}).
\end{remark}
\begin{definition}[$L^2$-strong compactness]
A compact metric measure space $(Y, \dist, \nu)$ is said to satisfy the \textit{$L^2$-strong compactness condition} if the canonical inclusion $\iota: H^{1, 2}(Y, \dist, \nu) \hookrightarrow L^2(Y, \nu)$ is a compact operator.
\end{definition}
It is well-known that there are several sufficient conditions to satisfy the $L^2$-strong compactness condition, for instance, PI-condition (i.e. the volume doubling and the Poincar\'e inequality are satisfied), which follows from $\RCD(K, N)$-conditions for $N<\infty$ (see for instance \cite{HK} for the proof of the $L^2$-strong compactness condition). However, in general, for an $n$-dimensional almost smooth compact metric measure space, the $L^2$-strong compactness condition is not satisfied even if $\Omega$ has a uniform lower Ricci bound. 
To see this, 
for any two pointed metric spaces $(X_i, \dist_i, x_i)(i=1, 2)$, let us denote by $(X_1, \dist_1, x_1) * (X_2, \dist_2, x_2)$ their glued pointed metric space as $x_1=x_2$, that is, the metric space is
$$
X_1*X_2:=(X_1 \bigsqcup X_2) /(x_1=x_2)
$$
with the intrinsic metric, and the base point is the glued point. See \cite{BuragoBuragoIvanov} for the detail.
Sometimes, we denote the metric space by $(X_1 *X_2, \dist)$ without any attention on the base points for simplicity.
\begin{example}\label{ddddd}
Let us define a sequence of pointed compact metric spaces $(X_i, \dist_i, x_i)$ as follows. 
Fix $n \ge 3$ and consider a sequence of flat $n$-tori:
$$
\mathbb{T}_i^n:=\mathbb{S}^1(1/2^i) \times \mathbb{S}^1(1/2^i) \times \cdots \times \mathbb{S}^1(1/2^i)
$$
with fixed points $p_i \in \mathbb{T}_i^n$, where $\mathbb{S}^1(r):=\{v \in \mathbb{R}^2;|v|=r\}$. Then, let $(X_1, \dist_1, x_1):=(\mathbb{T}_1^n, \dist_{\mathbb{T}^n_{1}}, p_1)$ and let
$$
(X_{i+1}, \dist_{i+1}, x_{i+1}):=(X_i, \dist_i, x_i) *(\mathbb{T}^n_{i+1}, \dist_{\mathbb{T}^n_{i+1}}, p_{i+1}) \quad \forall i \ge 1.
$$
Then, let us denote by $(X, \dist, x)$ the pointed Gromov-Hausdorff limit space of $(X_i, \dist_i, x_i)$. Note that $(X, \dist)$ is compact, that $\Omega:=X \setminus \{x\}$ satisfies the smoothness with $\mathrm{Ric}_{\Omega}^g \ge 0$, and that there exist canonical isometric embeddings $\mathbb{T}_i^n \hookrightarrow X$ (we identify $\mathbb{T}^n_i$ with the image). Then, we consider the $n$-dimensional Hausdorff measure $\mathcal{H}^n$ as the reference measure $\meas$ on $X$.

Let us check the zero capacity condition. It is trivial that $\mathcal{H}^n(X \setminus \Omega)=0$.
For all $\epsilon>0$, take $\psi_{\epsilon} \in C^{\infty}(\mathbb{R})$ satisfying that $\psi_{\epsilon}|_{(-\infty, \epsilon]} \equiv 0$, that $0\le \psi_{\epsilon} \le 1$,  that $\psi_{\epsilon}|_{[2\epsilon, \infty)} \equiv 1$, that $|\psi_{\epsilon}'| \le 100/\epsilon$ and that $|\psi_{\epsilon}''| \le 100/\epsilon^2$.

Define $\phi_i \in C^{\infty}_c(\Omega)$ by
$$
\phi_i(y):=\sum_{j=1}^i1_{\mathbb{T}^n_j}(y)\psi_{\pi /2^{i+10}}(\dist (x, y)).
$$
Then, since it is easy to see that for some universal constant $C_1>0$
$$
|\Delta \psi_{\pi /2^{i+10}}(\dist (x, \cdot))| (y) \le C_12^{2i} \quad \forall j \le i,  \forall y \in \mathbb{T}^n_j \cap \left(\overline{B}_{\pi/2^{i+9}}(x) \setminus B_{\pi /2^{i+10}}(x)\right),
$$
we see that for all $j \le i$
\begin{align*}
&\int_{\mathbb{T}^n_j}|\Delta \psi_{\pi /2^{i+10}}(\dist (x, y))|\dist \mathcal{H}^n \\
&=\int_{\mathbb{T}^n_j \cap \left(\overline{B}_{\pi/2^{i+9}}(x) \setminus B_{\pi /2^{i+10}}(x)\right)}|\Delta \psi_{\pi /2^{i+10}}(\dist (x, y))|\dist \mathcal{H}^n \le C_22^{(2-n)i},
\end{align*}
where $C_2$ is also a universal constant.
In particular, 
$$
\int_X|\Delta \phi_i|\dist \mathcal{H}^n \le C_2i2^{(2-n)i} \to 0 \quad (i \to \infty)
$$
which proves the zero capacity condition.
Thus, $(X, \dist, \mathcal{H}^n)$ is an $n$-dimensional almost smooth compact metric measure space.

Let us define a sequence $f_i \in L^2(X, \mathcal{H}^n)$ by
$$
f_i:=\frac{1}{\mathcal{H}^n(\mathbb{T}^n_i)}1_{\mathbb{T}^n_i}.
$$
Then, it is easy to see that $f_i$ $L^2$-weakly converge to $0$ and that $f_i \in H^{1, 2}(X, \dist, \meas)$ with $\|f_i\|_{L^2}=\|f_i\|_{H^{1, 2}}=1$  (see also Example \ref{ccd}).
Since $f_i$ does not $L^2$-strongly converge to $0$, the $L^2$-strong compactness condition does not hold.
\end{example}
It follows from standard arguments in functional analysis that if an infinitesimally Hilbertian compact metric measure space $(Y, \dist, \nu)$ satisfies the $L^2$-strong compactness condition with $\mathrm{dim}\,L^2(Y, \nu)=\infty$, then, the spectrum of $-\Delta$ is discrete and unbounded (each eigenvalue has finite multiplicities). Thus, we then denote the eigenvalues by
$$
0=\lambda_1(Y, \dist, \nu) \le \lambda_2(Y, \dist, \nu) \le \lambda_2(Y, \dist, \nu) \le \cdots \to \infty
$$
counted with multiplicities, and denote the corresponding eigenfunctions by $\phi_i^Y$ with $\|\phi_i^Y\|_{L^2}=1$. We always fix an $L^2$-orthogonal basis $\{\phi_i^Y\}_i$ consisting of eigenfunctions, immediately.
Moreover, it also holds that for all $f \in L^2(Y, \nu)$,
\begin{equation}\label{ee}
f=\sum_i \left(\int_Yf\phi_i^Y\dist \nu\right) \phi_i^Y\quad \mathrm{in}\,L^{2}(Y,  \nu)
\end{equation}
and that for all $f \in H^{1, 2}(Y, \dist, \nu)$,
\begin{equation}\label{ee2}
f=\sum_i \left(\int_Yf\phi_i^Y\dist \nu\right) \phi_i^Y\quad \mathrm{in}\,H^{1, 2}(Y, \dist, \nu).
\end{equation}
For reader's convenience, we will give proofs of them in the appendix.

We are now in a position to give the main result:
\begin{theorem}[From $\mathrm{Ric}_{\Omega}^g\ge K(n-1)$ to $\BE(K(n-1), n)$]\label{be}
Let $(X, \dist, \meas)$ be an $n$-dimensional almost smooth compact metric measure space.
Assume that $(X, \dist, \meas)$ satisfies the $L^2$-strong compactness condition, that each eigenfunction $\phi_i^X$
 satisfies $|\nabla \phi_i^X| \in L^{\infty}(X, \meas)$ and that $\mathrm{Ric}_{\Omega}^g\ge K(n-1)$ for some $K \in \mathbb{R}$. Then, $(X, \dist, \meas)$ satisfies the $\BE(K(n-1), n)$-condition.
\end{theorem}
\begin{proof}
Let us use the same notation as above, that is, let $f_N:=\sum_i^Na_i\phi_i^X$, where $a_i:=\int_Xf\phi_i^X\dist \meas$.
Note that by (\ref{ee2}), $f_N, \Delta f_N \to f, \Delta f$ in $H^{1, 2}(X, \dist, \meas)$, respectively as $N \to \infty$.
In the following, for all $h \in C^{\infty}(\Omega)$, the Laplacian $\mathrm{tr} (\mathrm{Hess}_h)$ defined in Riemannian geometry is also denoted by the same notation $\Delta h$, without any attention because
\begin{equation}\label{kkjjnnmmkk}
\int_{\Omega}\langle \nabla h, \nabla \psi\rangle \dist \mathcal{H}^n=-\int_{\Omega}\mathrm{tr} (\mathrm{Hess}_h)\psi \dist \mathcal{H}^n, \quad \forall \psi \in C^{\infty}_c(\Omega)
\end{equation}
is satisfied and (\ref{kkjjnnmmkk}) characterizes the function $\mathrm{tr} (\mathrm{Hess}_f)$ in $L^2_{\mathrm{loc}}(\Omega, \mathcal{H}^n)$.

Fix $N \in \mathbb{N}$. Then, let us prove that $|\nabla f_N|^2 \in H^{1, 2}(X, \dist, \meas)$ as follows.

By our assumption on the eigenfunctions, we see that $|\nabla f_N| \in L^{\infty}(X, \meas)$. Moreover, the elliptic regularity theorem shows that $f_N|_{\Omega} \in C^{\infty}(\Omega)$. 
Since $\mathrm{Ric}_{\Omega}^g\ge K(n-1)$, we have 
\begin{equation}\label{rr}
\frac{1}{2}\Delta |\nabla f_N|^2 \ge |\mathrm{Hess}_{f_N}|^2+\langle \nabla \Delta f_N, \nabla f_N\rangle +K(n-1)|\nabla f_N|^2 \quad \forall x \in \Omega.
\end{equation}
Thus, multiplying $\phi_i$ on both sides and then integrating this over $\Omega$ show
\begin{align}\label{gg}
&\frac{1}{2}\int_{\Omega}|\nabla f_N|^2\Delta \phi_i \dist \mathcal{H}^n \ge  \int_{\Omega}\phi_i\left( |\mathrm{Hess}_{f_N}|^2 +\langle \nabla \Delta f_N, \nabla f_N\rangle +K(n-1)|\nabla f_N|^2\right)\dist \mathcal{H}^n.
\end{align}
Since $|\nabla f_N| \in L^{\infty}(X, \meas)$ and our assumption on the zero capacity, the inequality (\ref{gg}) implies
$$
\limsup_{i \to \infty}\int_{\Omega}\phi_i|\mathrm{Hess}_{f_N}|^2\dist \mathcal{H}^n<\infty.
$$
In particular, 
$$
\int_A |\mathrm{Hess}_{f_N}|^2\dist \mathcal{H}^n\le \limsup_{i \to \infty}\int_{\Omega}\phi_i|\mathrm{Hess}_{f_N}|^2\dist \mathcal{H}^n<\infty \quad \forall A \Subset \Omega.
$$
Thus, the monotone convergence theorem yields
\begin{equation}\label{hhhj}
\int_{\Omega}|\mathrm{Hess}_{f_N}|^2\dist \mathcal{H}^n<\infty.
\end{equation}

On the other hand, since $\phi_i \in C_c^{\infty}(\Omega)$ and $f_N|_{\Omega} \in C^{\infty}(\Omega)$, we have $\phi_i |\nabla f_N|^2 \in H^{1, 2}(X, \dist, \meas)$. Moreover, since
\begin{align*}
\int_X|\nabla (\phi_i|\nabla f_N|^2)|^2\dist \meas & \le \int_X \left(2 |\nabla \phi_i|^2|\nabla f_N|^4 + 2|\nabla |\nabla f_N|^2|^2\right) \dist \meas \\
&\le \int_X \left(2 |\nabla \phi_i|^2|\nabla f_N|^4 + 2|\mathrm{Hess}_{f_N}|^2 |\nabla f_N|^2\right) \dist \meas,
\end{align*}
by (\ref{hhhj}), we have $\sup_i\|\phi_i|\nabla f_N|^2\|_{H^{1, 2}}<\infty$ which completes the proof of the desired statement,  $|\nabla f_N|^2 \in H^{1, 2}(X, \dist, \meas)$, because $\phi_i|\nabla f_N|^2 \to |\nabla f_N|^2$ in $L^2(X, \meas)$.

We are now in a position to finish the proof. Let $\phi \in D(\Delta) \cap L^{\infty}(X, \meas)$ with $\Delta \phi \in L^{\infty}(X, \meas)$ and $\phi \ge 0$.
Multiplying $\phi \phi_i$ on both sides of (\ref{rr}) and integrating this over $X$ show
\begin{align}\label{gg2}
-\frac{1}{2}\int_X\langle \nabla (\phi \phi_i), \nabla |\nabla f_N|^2\rangle\dist \meas \ge  \int_X\phi\phi_i \left(|\mathrm{Hess}_{f_N}|^2 +\langle \nabla \Delta f_N, \nabla f_N\rangle +K(n-1)|\nabla f_N|^2\right)\dist \meas.
\end{align}

Recall that $\phi_i \to 1$ in $L^2(X, \meas)$ and that $\nabla \phi_i$ $L^2$-weakly converge to $\nabla 1=0$ with $\nabla (\phi \phi_i)=\phi_i\nabla \phi +\phi\nabla \phi_i$. Thus, 
we have
$$
\mathrm{LHS}\,\mathrm{of}\, (\ref{gg2}) \to -\frac{1}{2}\int_X\langle \nabla \phi, \nabla |\nabla f_N|^2\rangle \dist \meas = \frac{1}{2}\int_X |\nabla f_N|^2\Delta \phi\dist \meas,
$$
where we used $\|\nabla |\nabla f_N|^2\|_{L^2}\le 2\|\mathrm{Hess}_{f_N}\|_{L^2}\|\nabla f_N\|_{L^{\infty}}<\infty$ and $|\nabla f_N|^2 \in H^{1, 2}(X, \dist, \meas)$.
Moreover, the dominated convergence theorem yields
$$
\mathrm{RHS}\,\mathrm{of}\, (\ref{gg2}) \to \int_X\phi \left(|\mathrm{Hess}_{f_N}|^2 +\langle \nabla \Delta f_N, \nabla f_N\rangle +K(n-1)|\nabla f_N|^2\right)\dist \meas.
$$
Thus, combining these with letting $i \to \infty$ in (\ref{gg2}) shows
\begin{align*}
\frac{1}{2}\int_X |\nabla f_N|^2\Delta \phi\dist \meas &\ge  \int_X\phi \left(|\mathrm{Hess}_{f_N}|^2 +\langle \nabla \Delta f_N, \nabla f_N\rangle +K(n-1)|\nabla f_N|^2\right)\dist \meas \\
&\ge \int_X\phi \left(\frac{(\Delta f_N)^2}{n} +\langle \nabla \Delta f_N, \nabla f_N\rangle +K(n-1)|\nabla f_N|^2\right)\dist \meas.
\end{align*}
Therefore, letting $N \to \infty$ yields
$$
\frac{1}{2}\int_X|\nabla f|^2 \Delta \phi \dist \meas  \ge \int_X\phi \left( \frac{(\Delta f)^2}{n} +\langle \nabla \Delta f, \nabla f\rangle +K(n-1)|\nabla f|^2\right)\dist \meas
$$
which completes the proof.
\end{proof}
Let us apply Theorem \ref{be} to an explicit simple example as follows.
\begin{example}\label{ccd}
Let us check that the metric measure space $$(X, \dist, \meas):=\left(\mathbb{S}^n(1) * \mathbb{S}^n(1), \dist, \mathcal{H}^n\right)$$ satisfies the $\BE(n-1, n)$-condition ($n \ge 2$), where $\mathbb{S}^n(r):=\{v \in \mathbb{R}^{n+1}; |v|=r\}$.
Let us denote by $\mathbb{S}^n_1(1)$ and $\mathbb{S}^m_2(1)$, respectively, the images of the canonical isometric embeddings $\mathbb{S}^n(1) \hookrightarrow X$ to the first sphere and the second one, respectively.
Moreover, we denote by $p$ the intersection point of them. It is worth pointing out that $(X, \dist, \meas)$ satisfies the Ahlfors $n$-regularity, which is easily checked.

\textit{Being an $n$-dimensional almost smooth compact metric measure space.} 
Let $\Omega:=X \setminus \{p\}$. Then, it is trivial that $\Omega$ satisfies the smoothness with $\mathrm{Ric}_{\Omega}^g\ge (n-1)$ and $\mathcal{H}^n(X \setminus \Omega)=0$.

Let us use $\psi_{\epsilon}$ as in Example \ref{ddddd}.
Then, by an argument similar to that in Example \ref{ddddd}, it is easy to check that the functions $\phi_i(x):=\psi_{i^{-1}}(\dist(p, x))$ satisfies the zero capacity condition. Thus, $(X, \dist, \meas)$ is an $n$-dimensional almost smooth compact metric measure space.

\textit{Satisfying the $L^2$-strong compactness condition.}
We remark that 
\begin{equation}\label{gvv}
f1_{\mathbb{S}^n_j(1)} \in H^{1, 2}(X, \dist, \meas) \quad \forall f \in H^{1, 2}(X, \dist, \meas)
\end{equation}
and 
\begin{equation}\label{qqqqq2}
h1_{\mathbb{S}^n_j(1)} \in H^{1, 2}(X, \dist, \meas) \quad \forall h \in H^{1, 2}(\mathbb{S}^n_j(1), \dist, \mathcal{H}^n)
\end{equation}
are satisfied,
which come from the zero capacity condition with the truncation argument, that is, for functions $\phi_i1_{\mathbb{S}^n_j(1)}(f\wedge L \vee -L)$ and $\phi_i1_{\mathbb{S}^n_j(1)}(h\wedge L\vee -L)$, letting $i \to \infty$ and then letting $L \to \infty$ show (\ref{gvv}) and (\ref{qqqqq2}) (recall the proof of (2) of Proposition \ref{fun}).

Let $f_i \in H^{1, 2}(X, \dist, \meas)$ with $\sup_i\|f_i\|_{H^{1, 2}}<\infty$.  Put $f_i^j:=f_i1_{\mathbb{S}^n_j(1)} \in H^{1, 2}(\mathbb{S}^n_j(1), \dist, \mathcal{H}^n)$. Then, since the $L^2$-strong compactness condition holds for $(\mathbb{S}^n_j(1), \dist, \mathcal{H}^n)$, there exist a subsequence $i(k)$ and $f^j \in L^2(\mathbb{S}^n_j(1), \mathcal{H}^n)$ such that $f_{i(k)}^j \to f^j$ in $L^2(\mathbb{S}^n_j(1), \mathcal{H}^n)$ for all $j \in \{1, 2\}$.
In particular, $f_{i(k)}=f_{i(k)}^1+f_{i(k)}^2 \to f^1+f^2=:f$ in $L^2(X, \meas)$, which proves the $L^2$-strong compactness condition for $(X, \dist, \meas)$.

\textit{Satisfying the gradient estimates on the eigenfucntions and the $\BE(n-1, n)$ condition.}
Let $f$ be an eigenfunction, that is, $f \in D(\Delta)$ with $-\Delta f=\lambda f$ for some $\lambda \ge 0$. For any $h \in H^{1, 2}(\mathbb{S}^n_j(1), \dist, \mathcal{H}^n)$, put $\hat{h}=h1_{\mathbb{S}^n_j(1)} \in H^{1, 2}(X, \dist, \meas)$. Then, since
\begin{equation}\label{aasa}
\int_X\langle \nabla f, \nabla \hat{h}\rangle \dist \meas =\lambda \int_Xf\hat{h}\dist \meas
\end{equation}
and 
$$
\mathrm{LHS}\,\mathrm{of}\,(\ref{aasa}) = \int_{\mathbb{S}^n_j(1)}\langle \nabla f, \nabla h\rangle \dist \mathcal{H}^n, \quad \mathrm{RHS}\,\mathrm{of}\,(\ref{aasa}) = \lambda \int_{\mathbb{S}^n_j(1)}f h \dist \mathcal{H}^n, 
$$
we see that
$f|_{\mathbb{S}^n_j(1)}$ is an eigenfunction of $(\mathbb{S}^n_j(1), \dist, \mathcal{H}^n)$. Thus, $|\nabla (f|_{\mathbb{S}^n_j(1)})| \in L^{\infty}(\mathbb{S}_j^n(1), \mathcal{H}^n)$, which implies $|\nabla f| \in L^{\infty}(X, \meas)$.

Therefore, we can apply Theorem \ref{be} to show that $(X, \dist, \meas)$ satisfies the $\BE(n-1, n)$-condition.

\textit{Coincidence between the induced distance $\dist_{\Ch}$ by the Cheeger energy and $\dist$.} 
Let us prove:
\begin{equation}\label{kkm}
\dist_{\Ch}(x, y)= \dist (x, y), \quad \forall x, y \in X,
\end{equation}
where
\begin{equation}\label{honn}
\dist_{\Ch}(x, y):=\sup \left\{\phi (x)-\phi (y); \phi \in C^0(X) \cap H^{1, 2}(X, \dist, \meas), |\nabla \phi |(z) \le 1, \meas-a.e. z \in X\right\}. 
\end{equation}
Let $x \in \mathbb{S}^n_1(1)$ and let $y \in \mathbb{S}^n_2(1)$.
For any $\phi$ as in the RHS of (\ref{honn}), 
\begin{align}\label{hfhf}
\phi (x)-\phi(y) &= \phi|_{\mathbb{S}^n_1(1)}(x)-\phi|_{\mathbb{S}^n_1(1)}(p) + \phi|_{\mathbb{S}^n_2(1)}(p)-\phi|_{\mathbb{S}^n_2(1)}(y) \nonumber \\
&\le \dist_{\mathbb{S}^n_1(1)}(x, p)+\dist_{\mathbb{S}^n_2(1)}(p, y)=\dist (x, y),
\end{align}
where we used the fact that $\dist=\dist_{\Ch}$ in $(\mathbb{S}^n_j(1), \dist_{\mathbb{S}^n_j(1)}, \mathcal{H}^n)$. Thus, taking the supremum in (\ref{hfhf}) with respect to $\phi$ shows the inequality `$\le$' in (\ref{kkm}).

To get the converse inequality, let 
$$
\phi(z):=(1_{\mathbb{S}^n_1(1)}(z)-1_{\mathbb{S}^n_2(1)}(z))\dist (p, z).
$$
Then, we see that $\phi \in \mathrm{LIP}(X, \dist)$, that $\mathrm{Lip} \phi (z)\le 1$ for all $z \in X$, and that $\phi(x)-\phi(y)=\dist (x, y)$, which proves the converse inequality `$\ge$' in (\ref{kkm}).

Similarly, we can prove (\ref{kkm}) in the remaining case, thus, we have (\ref{kkm}) for all $x, y \in X$.

\textit{Poincar\'e inequality and $\RCD (K, \infty)$ condition are not satisfied.}
Assume that $(X, \dist, \meas)$ satisfies the $(1, 2)$-Poincar\'e inequality, that is, there exists $C>0$ such that for all $r>0$, all $x \in X$ and all $f \in H^{1, 2}(X, \dist, \meas)$, it holds that
\begin{equation}\label{ll}
\frac{1}{\meas (B_r(x))}\int_{B_r(x)}\left| f- \frac{1}{\meas (B_r(x))}\int_{B_r(x)}f\dist \meas \right|\le Cr\left(\frac{1}{\meas (B_r(x))}\int_{B_r(x)}|\nabla f|^2\dist \meas\right)^{1/2}.
\end{equation}
Let $\phi:=1_{\mathbb{S}^n_1(1)}-1_{\mathbb{S}^n_2(1)}$. By (\ref{qqqqq2}), we have $\phi \in H^{1, 2}(X, \dist, \meas)$.
Then, by the locality of the slope, we have $|\nabla \phi|=0$ $\meas$-a.e.. In particular, (\ref{ll}) yields that $\phi$ must be a constant, which is a contradiction.

By the same reason, for all $K \in \mathbb{R}$, $(X, \dist, \meas)$ does not satisfy $\RCD(K, \infty)$-condition. 
\end{example}
\begin{remark}\label{yy6}
Example \ref{ccd} also tells us that (\ref{22}) does not imply the expected Bishop-Gromov inequalty.
In fact, under the same notation as in Example \ref{ccd}, letting $q$ be the antipodal point of $p$ in $\mathbb{S}^n_1(1)$ yields
$$
\frac{\meas (B_{2\pi}(q))}{\meas (B_\pi(q))}=2 >1 =\frac{\mathcal{H}^n(B_{2\pi}^{\mathbb{S}^n(1)}(x))}{\mathcal{H}^n(B_{\pi}^{\mathbb{S}^n(1)}(x))} \quad \forall x \in \mathbb{S}^n(1)
$$
which is the `reverse' Bishop-Gromov inequality.
Similarly, the $\BE(n-1, n)$ condition with `$\dist=\dist_{\Ch}$' does not imply the expected Bonnet-Myers theorem.
\end{remark}
\begin{corollary}[Characterization of $\RCD$ condition on almost smooth compact metric measure space]\label{corrcd}
Let $(X, \dist, \meas)$ be an $n$-dimensional almost smooth compact metric measure space associated with an open subset $\Omega$ of $X$, and let $K \in \mathbb{R}$.
Then, $(X, \dist, \meas)$ is a $\RCD(K(n-1), n)$ space if and only if the following four conditions hold:
\begin{enumerate}
\item the Sobolev to Lipschitz property holds;
\item the $L^2$-strong compactness condition holds; 
\item any eigenfunction is Lipschitz;
\item $\mathrm{Ric}_{\Omega}^g\ge K(n-1)$ holds.
\end{enumerate}
\end{corollary}
\begin{proof}
By Theorem \ref{be}, it is enough to check `only if' part. 
If $(X, \dist, \meas)$ is a $\RCD (K(n-1), n)$, then, applying Gigli's Bochner inequality (for $H^{1, 2}_H$-vector fields) in \cite{Gigli} shows that for all $f, h, \phi \in C^{\infty}_c(\Omega)$ with $\phi \ge 0$,
$$
\int_X\frac{\phi}{2}\Delta |f\nabla h|^2 \ge \int_X\phi\left( |\nabla (f\nabla h)|^2 -\langle \Delta_{H}(f\dist h), f\dist h\rangle +K(n-1)|f\nabla h|^2\right)\dist \meas
$$
which implies 
\begin{equation}\label{ttf}
\frac{1}{2}\Delta |f\nabla h|^2 \ge |\nabla (f\nabla h)|^2 -\langle \Delta_{H}(f\dist h), f\dist h\rangle +K(n-1)|f\nabla h|^2 \quad \forall x \in \Omega
\end{equation}
because $\phi$ is arbitrary, where $\Delta_H:=\dist \delta +\delta \dist$ is the Hodge Laplacian acting on $1$-forms.
In particular, since (\ref{ttf}) is equivalent to $\mathrm{Ric}_{\Omega}^g(f\nabla h, f\nabla h) \ge K(n-1)|f\nabla h|^2$ for all $x \in \Omega$, we have $\mathrm{Ric}_{\Omega}^g \ge K(n-1)$ because $f, h$ are also arbitrary.

The Sobolev to Lipschitz property is in the definition of $\RCD$ space. Moreover, as written previously, the $L^2$-strong compactness condition follows from the doubling condition and the Poincar\'e inequality, which are justified by the Bishop-Gromov inequality \cite{Sturm06} and by \cite{Rajala}.  
Finally, since the Lipschitz regularity on the eigenfunctions is satisfied by \cite{Jiang}, we conclude.
\end{proof}
\begin{corollary}[Another characterization of $\RCD$ condition]
Let $(X, \dist, \meas)$ be an $n$-dimensional almost smooth compact metric measure space associated with an open subset $\Omega$ of $X$, and let $K \in \mathbb{R}$.
Then, $(X, \dist, \meas)$ is a $\RCD(K(n-1), n)$ space if and only if the following four conditions hold:
\begin{enumerate}
\item $(X, \dist, \meas)$ is a PI space;
\item the induced distance $\dist_{\Ch}$ by the Cheeger energy is equal to the original distance $\dist$; 
\item any eigenfunction $f$ satisfies $|\nabla f| \in L^{\infty}(X, \meas)$;
\item $\mathrm{Ric}_{\Omega}^g\ge K(n-1)$ holds.
\end{enumerate}
\end{corollary}
\begin{proof}
Since the proof of `only if' part is same to that of Corollary \ref{corrcd}, let us check `if' part.
By Theorem \ref{be}, we see that $(X, \dist, \meas)$ is a $\BE(K(n-1), n)$ space.
Thus, it suffices to check the Sobolev to Lipschitz property.

Let $f \in H^{1, 2}(X, \dist, \meas)$ with $|\nabla f|(x) \le 1$ $\meas$-a.e. $x \in X$.
Then, the telescope argument with the PI condition (c.f. \cite{Cheeger}) yields that there exists $\hat{f} \in \mathrm{LIP}(X, \dist)$ such that $f(x)=\hat{f}(x)$ $\meas$-a.e. $x \in X$.
Then, since it is proved in \cite{AmbrosioErbarSavare} that 
\begin{itemize}
\item any $h \in H^{1, 2}(X, \dist, \meas) \cap C^0(X)$ with $|\nabla h| \le 1$ $\meas$-a.e. $x \in X$ is $1$-Lipschitz,
\end{itemize}
we see that $\hat{f}$ is $1$-Lipschitz, that is, the Sobolev-Lipschitz property holds.
Thus, we conclude.  
\end{proof}
\begin{remark}
We should mention that very recently, similar characterization of $\RCD$ conditions for \textit{stratified spaces}, which give almost smooth metric measure spaces as typical examples, is proved in \cite{BKMR}.
\end{remark}
We end this section by giving a sufficient condition to satisfy the Sobolev to Lipschitz property. For that, let us introduce the definition of the segment inequality:
\begin{definition}[Segment inequality]
Let $(Y, \dist, \nu)$ be a metric measure space satisfying that $(Y, \dist)$ is a geodesic space.
For a nonnegative valued Borel function $f$ on $Y$, define 
$$
\mathcal{F}_f(x, y) := \inf_{\gamma}\int_{[0, \dist (x, y)]}f(\gamma)\dist s, \quad \forall x, y \in Y,
$$
where the infimum runs over all minimal geodesics $\gamma$ from $x$ to $y$.
Then, we say that 
\textit{$(Y, \dist, \nu)$ satisfies the segment inequality} if there exists $\lambda>0$ such that 
$$\int_{B_r(x) \times B_r(x)}\mathcal{F}_f(y, z) \dist(\nu \times \nu) \le \lambda r \nu (B_r(x)) \int_{B_{\lambda r}(x)}f\dist \nu \quad \forall x \in Y, \forall r>0, \forall f.
$$
\end{definition}
Cheeger-Colding proved in \cite{CheegerColding3} that if $(Y, \dist, \nu)$ satisfies the volume doubling condition and the segment inequality, then, the $(1, 1)$-Poincar\'e inequality holds (see also \cite{CheegerColding} and \cite{HP}).
\begin{proposition}[Segment inequality with doubling condition implies Sobolev to Lipschitz property]
Let $(Y, \dist, \nu)$ be a compact metric measure space satisfying that $(Y, \dist)$ is a geodesic space.
Assume that $(Y, \dist, \nu)$ satisfies the volume doubling condition and the segment inequality. Then, the Sobolev to Lipschitz property holds.
\end{proposition}
\begin{proof}
Let $f \in H^{1, 2}(Y, \dist, \nu)$ with $|\nabla f| \le 1$ $\nu$-a.e.. As written above, since $(1, 1)$-Poincar\'e inequality is satisfied, $f$ has a representative in $\hat{f} \in \mathrm{LIP}(Y, \dist)$ by the telescope argument (see for instance \cite{Cheeger}).
Thus, since we have $|\mathrm{Lip}\hat{f}| \le 1$ $\nu$-a.e., which also follows from \cite{Cheeger}, it suffices to check that $\hat{f}$ is $1$-Lipschitz.

Let us take a Borel subset $A$ of $Y$ such that $\nu(Y \setminus A)=0$ and that $$\mathrm{Lip}\hat{f} \le 1, \quad \forall x \in A.$$
Applying the segment inequality for $1_{Y \setminus A}$ yields that there exists a Borel subset $B$ of $Y \times Y$ such that $(\nu \times \nu ) ((Y \times Y) \setminus B)=0$ and that for any $(x, y) \in B$ and any $\epsilon >0$, there exists a minimal geodesic $\gamma$ from $x$ to $y$ such that 
$$
\int_{[0, \dist (x, y)]}1_{Y \setminus A}(\gamma (s))\dist s<\epsilon.$$

Therefore, since $\mathrm{Lip}\hat{f}$ is an upper gradient of $\hat{f}$, we have
\begin{align}\label{21212121}
|\hat{f}(x)-\hat{f}(y)| &\le \int_{[0, \dist (x, y)]}\mathrm{Lip}\hat{f}(\gamma (s))\dist s \nonumber \\
&= \int_{[0, \dist (x, y)]} 1_{A}(\gamma (s))\mathrm{Lip}\hat{f}(\gamma (s))\dist s + \int_{[0, \dist (x, y)]} 1_{Y \setminus A}(\gamma (s))\mathrm{Lip}\hat{f}(\gamma (s))\dist s \nonumber \\
&\le \dist (x, y) + \sup_z \mathrm{Lip}\hat{f}(z) \epsilon.
\end{align}
Since $\epsilon$ is arbitrary and $B$ is dense in $Y \times Y$, (\ref{21212121}) yields that $\hat{f}$ is $1$-Lipschitz.
\end{proof}
\begin{remark}Let us give remarks on related works. Note that in the following, the spaces are not necessary compact. As we already used, Jiang proved in \cite{Jiang} the gradient estimates on solutions of Poisson's equations (including eigenfunctions) in the setting of metric measure spaces under assuming mild geometric conditions and a heat semigroup curvature condition (or called an weighted Sobolev inequality). Bamler and Chen-Wang proved in \cite{Bamler}, in \cite{ChenWang}, such conditions (including the segment inequality) in their almost smooth settings, independently.

One of interesting questions is; if an $n$-dimensional almost smooth (compact) metric measure space $(X, \dist, \meas)$ satisfies that the induced distance $\dist_g$ by $g$ on $\Omega$ coincides with $\dist|_{\Omega}$, then, being $\mathrm{Ric}_{\Omega}^g\ge K(n-1)$ is equivalent to that $(X, \dist, \meas)$ is a $\RCD(K(n-1), n)$-space?
\end{remark}
\subsection{Nonconstant dimensional case}
In this section, let us discuss a variant of $n$-dimensional almost smooth compact metric measure spaces. Let us recall that we fix a compact metric measure space $(X, \dist, \meas)$.
\begin{definition}[Generalized almost smooth compact metric measure space]
We say that $(X, \dist, \meas)$ is a \textit{generalized almost smooth compact metric measure space associated with an open subset $\Omega$ of $X$} if the following two conditions are satisfied;
\begin{enumerate}
\item{(Generalized smoothness of $\Omega$)} for all $p \in \Omega$, there exist an integer $n(p) \in \mathbb{N}$, an open neighborhood $U_p$ of $p$ in $\Omega$, an $n(p)$-dimensional (possibly incomplete) Riemannian manifold $(M^{n(p)}, g)$ and a map $\phi: U_p \to M^{n(p)}$ such that $\phi$ is a local isometry between $(U_p, \dist)$ and $(M^{n(p)}, \dist_g)$;
\item{(Hausdorff measure condition)} For all $p \in \Omega$, take $U_p$ as above. Then, the restriction of $\meas$ to $U_p$ coincides with the $n(p)$-dimensional Hausdorff measure $\mathcal{H}^{n(p)}$; 
\item{(Zero capacity condition)} $X \setminus \Omega$ has zero capacity in the sense of Definition \ref{def:asmm}.
\end{enumerate}
\end{definition}
By an argument similar to the proof of Theorem \ref{be}, we have the following:
\begin{theorem}[From $\mathrm{Ric}_{\Omega}^g \ge K$ to $\BE (K, N)$]
Let $(X, \dist, \meas)$ be a generalized almost smooth compact metric measure space associated with an open subset $\Omega$ of $X$.
Assume that $(X, \dist, \meas)$ satisfies the $L^2$-strong compactness condition, that each eigenfunction $\phi_i$
 satisfies $|\nabla \phi_i| \in L^{\infty}(X, \meas)$ and that $\mathrm{Ric}_{\Omega}^g\ge K$ for some $K \in \mathbb{R}$. Then, $(X, \dist, \meas)$ satisfies the $\BE(K, N)$-condition, where $N:=\sup_pn(p)$.
\end{theorem}
\begin{example}\label{1091}
By an argument similar to that in Example \ref{ccd}, we can easily see that for any two (not necessary same dimensional) closed pointed Riemannian manifolds $(M_i^{m_i}, g_i, p_i) (m_i \ge 2)$ with $\mathrm{Ric}_{M_i^{m_i}}^{g_i} \ge K$ for some $K \in \mathbb{R}$,  the metric measure space 
$$
\left(M_1^{m_1} * M_2^{m_2}, \dist, \mathcal{H}^{m_1}\res_{M_1^{m_1}} + \mathcal{H}^{m_2}\res_{M_2^{m_2}}\right)
$$
is a $\BE(K, \max \{m_1, m_2\})$ space with $\dist_{\Ch}=\dist$. 

More generally, similar constructions of $\BE(K, N)$ spaces by gluing embedded closed convex submanifolds $N_i^{n_i} \subset M_i^{m_i}$, which are isometric to each other, with $m_i-n_i\ge 2$ are also justified. 
\end{example}
\begin{remark} In this paper we discuss only the \textit{unweighted} case, that is, the restriction of the reference measure to the smooth part is the optimal Hausdorff measure $\mathcal{H}^n$.
Similar results are also obtained in the weighted case, $e^{-f}\dist \mathcal{H}^n$, where $f \in C^{\infty}(\Omega)$, under suitable assumptions on $f$ by using the Bakry-\'Emery ($N$-) Ricci tensor (and the Witten Laplacian $\Delta_f$, respectively) instead of using the original Ricci tensor (and the Laplacian $\Delta$, respectively).
However, we do not discuss the details because our main forcus is to discuss some `flexibility' on the $\BE(K, N)$ conditions as in Example \ref{1091}, which is very different from the $\RCD(K, N)$ conditions, and to give a bridge between almost smooth spaces and noncollapsed $\RCD$ spaces introduced in \cite{PhGi}, as in Corollary \ref{corrcd}.
\end{remark}
\section{Appendix}

Let $(Y, \dist, \nu)$ be an infinitesimally Hilbertian compact metric measure space and assume that $(Y, \dist, \nu)$ satisfies the $L^2$-strong compactness condition with $\dim L^2(Y, \nu) =\infty$. In this appendix, we will show that the spectrum of $-\Delta$ is discrete and unbounded, and that (\ref{ee}) and (\ref{ee2}) hold. 

Let us begin with proving the following lemma:
\begin{lemma}\label{ww3}
For all $\lambda \in \mathbb{R}$, let $E(\lambda):=\{ f \in D(\Delta);-\Delta f=\lambda f\}$.
\begin{enumerate}
\item If $\dim E(\lambda) \ge 1$, then $\lambda \ge 0$ (which is called an eigenvalue of $-\Delta$).
\item $\dim E(\lambda)<\infty$ holds.
\end{enumerate}
\end{lemma}
\begin{proof}
Let us check (1). Taking $f \in E(\lambda)$ with $f \neq 0$ in $L^2(Y, \nu)$ yields
$
\lambda =\frac{\int_Y|\nabla f|^2\dist \nu}{\int_Y|f|^2\dist \nu} \ge 0.
$

To prove (2), with no loss of generality, we can assume $\dim E(\lambda) \ge 1$.
Let us check $(E(\lambda), \| \cdot \|_{L^2})$ is a Hilbert space.
Take a Cauchy sequence $f_i$ in $E(\lambda)$. Let $f \in L^2(Y, \nu)$ be the $L^2$-strong limit function.
Since $\|\nabla f_i\|_{L^2}^2=\lambda \|f_i\|_{L^2}^2$, $f_i$ is a bounded sequence in $H^{1, 2}(Y, \dist, \nu)$. Thus, Mazur's lemma shows that $f \in H^{1, 2}(Y, \dist, \nu)$ and that $f_i$ converge weaky to $f$ in $H^{1, 2}(Y, \dist, \nu)$.
Therefore, letting $i \to \infty$ in
$$
\int_Y\langle \nabla f_i, \nabla g\rangle \dist \nu=\lambda \int_Yf_ig\dist \nu \quad \forall g \in H^{1, 2}(Y, \dist, \nu)
$$
yields
$$
\int_Y\langle \nabla f, \nabla g\rangle \dist \nu=\lambda \int_Yfg\dist \nu \quad \forall g \in H^{1, 2}(Y, \dist, \nu)
$$
which shows $f \in E(\lambda)$, where the convergence of the left hand sides comes from the polarization. Thus, $(E(\lambda), \| \cdot \|_{L^2})$ is a Hilbert space.

Then, similar argument with the $L^2$-strong compactness condition allows us to prove that $S(\lambda)$ is a compact subset of $E(\lambda)$, where $S(\lambda):=\{f \in E(\lambda); \|f\|_{L^2}=1\}$. Therefore, $\dim E(\lambda)<\infty$. 
\end{proof} 
\begin{lemma}\label{ww2}
The set $\mathcal{E}(Y, \dist, \nu):=\{\lambda \in \mathbb{R}_{\ge 0}; \dim E(\lambda) \ge 1\}$ is discrete.
\end{lemma}
\begin{proof}
The proof is done by contradiction. Assume that there exists a sequence $\lambda_i \in \mathcal{E}(Y, \dist, \nu)$ such that $\lambda_i \neq \lambda_j (i \neq j)$ and that $\lambda_i \to \lambda \in \mathbb{R}$.  
Take $f_i \in E(\lambda_i)$ with $\|f_i\|_{L^2}=1$. Then, since $\|f_i\|_{H^{1, 2}}^2=\lambda_i$, by the $L^2$-strong compactness condition, with no loss of generality, we can assume that there exists the $L^2$-strong limit function $f$ of $f_i$. Thus, $\|f\|_{L^2}=1$. 
Moreover, similar argument as in the proof of (2) of Lemma \ref{ww3} shows $f \in E(\lambda)$. In particular, $\lambda$ is an eigenvalue of $-\Delta$. Let $\{g_j\}_{j=1, 2, \ldots, N}$ be an ONB of $E(\lambda)$. Since $g_j \perp f_i$ in $L^2(Y, \nu)$, letting $i \to \infty$ yields $g_j \perp f$. Therefore, $\{g_j\}_j \cup \{f\}$ are linearly independent in $E(\lambda )$, which contradicts that $\{g_j\}_j$ is a basis of $E(\lambda)$.
\end{proof}
\begin{lemma}\label{ww4}
The set $\mathcal{E}(Y, \dist, \nu)$ is unbounded.
\end{lemma}
\begin{proof}
Note that since $1 \in E(0)$, we have $\mathcal{E}(Y, \dist, \nu) \neq \emptyset$.

Assume that $\mathcal{E}(Y, \dist, \nu)$ is bounded. 
Then, Lemma \ref{ww2} yields that $\mathcal{E}(Y, \dist, \nu)$ is a finite set.
By Lemma \ref{ww3}, there exists an ONB, $\{f_i\}_{i=1, 2,\ldots, N}$, of
$
\bigoplus_{\lambda \in \mathcal{E}(Y, \dist, \nu)}E(\lambda) (=:V).
$

On the other hand, it is easy to see that the number
$$
\lambda_*:=\inf_{f \perp V}\frac{\int_Y|\nabla f|^2\dist \nu}{\int_Y|f|^2\dist \nu}
$$
is also an eigenvalue of $-\Delta$ and that there exists a minimizer $f_*$ of the right hand side with $f_* \in E(\lambda_*)$ and $\|f_*\|_{L^2}=1$, where we used our assumption, $\dim L^2(Y, \nu)=\infty$, to make sence in the infimum. 
Thus, since $f_* \in V$ and $f_* \perp V$, we have $f_*=0$, which contradicts that $\|f_*\|_{L^2}=1$.
\end{proof}
Lemmas \ref{ww2} and \ref{ww4} allow us to denote the eigenvalues of $-\Delta$ by
$$
0=\lambda_1(Y, \dist, \nu) \le \lambda_2(Y,\dist, \nu) \le \cdots \to \infty
$$
counted with multiplicities.
Fix the corresponding eigenfunctions by $\phi_i^Y$ with $\|\phi_i^Y\|_{L^2}=1$.
\begin{proposition}\label{ww5}
For all $f \in L^2(Y, \nu)$, we have (\ref{ee}).
\end{proposition}
\begin{proof}
We first assume that $f \in H^{1, 2}(Y, \dist, \nu)$.
For all $N \in \mathbb{N}$, let $f_N:=\sum_i^N(\int_Yf\phi_i^Y\dist \nu)\phi_i^Y$ and let $g_N:=f-f_N$.
With no loss of generality, we can assume that $g_N \not \equiv 0$ for all $N$.

Then, since for all $i \le N$
\begin{align*}
\int_Yg_N\phi_i^Y\dist \nu&=\int_Yf\phi_i^Y\dist \nu-\sum_j^N\left(\int_Yf\phi_j^Y\dist \nu\right)\int_Y\phi_j^Y\phi_i^Y\dist \nu \\
&=\int_Yf\phi_i^Y\dist \nu-\int_Yf\phi_i^Y\dist \nu=0,
\end{align*}
we have $g_N \perp V_N$, where $V_N:= \mathrm{span}\,\{\phi_i^Y\}_{i=1, 2,\ldots, N}$.

On the other hand, it is easy to see that the number
$$
\lambda_{N+1}:=\inf_{h \perp V_N}\frac{\int_Y|\nabla h|^2\dist \nu}{\int_Y|h|^2\dist \nu}
$$
coincides with $\lambda_{N+1}(Y, \dist, \nu)$ (the inequality $\lambda_{N+1} \le \lambda_{N+1}(Y, \dist, \nu)$ is trivial. The converce is done by checking that $\lambda_{N+1}$ is an eigenvalue of $-\Delta$, which is similar to the proof of Lemma \ref{ww4}).

Therefore, we have $\|\nabla g_N\|_{L^{2}}^2 \ge \lambda_{N+1}(Y, \dist, \nu)\|g_N\|_{L^2}^2$.
Since
\begin{align*}
\int_Y|\nabla g_N|^2\dist \nu &= \int_Y|\nabla f|^2\dist \nu -2\int_Y\langle \nabla f, \nabla f_N\rangle \dist \nu +\int_Y|\nabla f_N|^2\dist \nu \\
&=\int_Y|\nabla f|^2\dist \nu -\int_Y|\nabla f_N|^2\dist \nu \le \int_Y|\nabla f|^2\dist \nu,
\end{align*}
we have $\|g_N\|_{L^2}^2\le (\lambda_{N+1}(Y, \dist, \nu))^{-1}\|\nabla f\|_{L^2}^2 \to 0$ as $N \to \infty$, which shows (\ref{ee}) in the case when $f \in H^{1, 2}(Y, \dist, \nu)$.

Next, let us check (\ref{ee}) for general $f \in L^2(Y, \nu)$.
Take a sequence $F_n \in H^{1, 2}(Y, \dist, \nu)$ with $F_n \to f$ in $L^2(Y, \nu)$.
Let $a_i:=\int_Yf\phi_i^Y\dist \nu$, let $a_{n, i}:=\int_YF_n\phi_i^Y\dist \nu$, let $f_N:=\sum_i^Na_i\phi_i^Y$ and let $F_{n, N}:=\sum_i^Na_{n, i}\phi_i^Y$.
For all $\epsilon>0$, there exists $n_0$ such that  $\|f-F_{n_0}\|_{L^2}<\epsilon$.
Then, there exists $N_0$ such that for all $N \ge N_0$, we have $\|F_{n_0}-F_{n_0, N}\|_{L^2}<\epsilon$.
Moreover, 
\begin{align*}
\int_Y|F_{n_0, N}-f_N|^2\dist \nu &=\sum_i^N(a_{n_0, i}-a_i)^2\\
&\le \sum_i(a_{n_0, i}-a_i)^2 \\
&= \sum_i\left(\int_Y(F_{n_0}-f_i)\phi_i^Y\dist \nu\right)^2 \le \|F_{n_0}-f\|_{L^2}^2 \le \epsilon^2,
\end{align*}
where we used the fact that for any ONS, $\{e_i\}_i$, in a Hilbert space $(H, \langle \cdot, \cdot \rangle)$, we have $|v|^2 \ge \sum_i\langle v, e_i\rangle^2$ for all $v \in H$.
Therefore, for all $N \ge N_0$,
$$
\|f-f_N\|_{L^2} \le \|f-F_{n_0}\|_{L^2}+\|F_{n_0}-F_{n_0, N}\|_{L^2}+\|F_{n_0, N}-f_N\|_{L^2} \le 3\epsilon,
$$
which completes the proof.
\end{proof}
\begin{proposition}
For all $f \in H^{1, 2}(Y, \dist, \nu)$, we have (\ref{ee2}).
\end{proposition}
\begin{proof}
Let $f_N:=\sum_i^Na_i\phi_i^Y$, where $a_i=\int_Yf\phi_i^Y\dist \nu$.
Then, 
$$
\|\nabla f_N\|_{L^2}^2=\sum_i^N\lambda_i(Y, \dist, \nu)(a_i)^2.
$$
On the other hand,
\begin{align*}
\int_Y\langle \nabla f, \nabla f_N\rangle \dist \nu &=\sum_i^Na_i\int_Y\langle \nabla f, \nabla \phi_i^Y\rangle \dist \nu \\
&=\sum_i^Na_i\lambda_i(Y, \dist, \nu)\int_Yf\phi_i^Y\dist \nu =\sum_i^N\lambda_i(Y, \dist, \nu)(a_i)^2 =\|\nabla f_N\|_{L^2}^2.
\end{align*}
In particular, the Cauchy-Schwartz inequality yields $\|\nabla f_N\|_{L^2}^2\le \|\nabla f\|_{L^2}\|\nabla f_N\|_{L^2}$. Thus, since $\|\nabla f_N\|_{L^2}\le \|\nabla f\|_{L^2}$, we have $\sup_i \|f_N\|_{H^{1, 2}}<\infty$. Since $f_N \to f$ in $L^2(Y, \nu)$ as $N \to \infty$, Mazur's lemma shows (\ref{ee2}).
\end{proof}


\begin{thebibliography}{GMS13}
\bibitem[AGS14a]{AmbrosioGigliSavare13}
	\textsc{L. Ambrosio, N. Gigli, G. Savar\'e}:
	\textit{Calculus and heat flow in metric measure spaces and applications to spaces with Ricci bounds from below}.
	Invent. Math. \textbf{195} (2014), 289--391.

\bibitem[AGS14b]{AmbrosioGigliSavare14}
	\textsc{L. Ambrosio, N. Gigli, G. Savar\'e}:
	\textit{Metric measure spaces with Riemannian Ricci curvature bounded from below}.
	Duke Math. J. \textbf{163} (2014), 1405--1490.
	    
 \bibitem[AGS15]{AmbrosioGigliSavare15}
       \textsc{L. Ambrosio, N. Gigli, G. Savar\'e}:
       \textit{Bakry-\'Emery curvature-dimension condition and Riemannian Ricci curvature bounds}.
       Ann. of Prob. \textbf{43} (2015), 339--404.
       
\bibitem[AGMR15]{AmbrosioGigliMondinoRajala}
        \textsc{L. Ambrosio, N. Gigli, A. Mondino, T. Rajala}:
        \textit{Riemannian Ricci curvature lower bounds in metric measure spaces with $\sigma$-finite measure.} 
Trans. of the AMS. \textbf{367} (2015), 4661--4701.

\bibitem[AES16]{AmbrosioErbarSavare}
        \textsc{L. Ambrosio, M. Erbar, G. Savar\'e}:
        \textit{Optimal transport, Cheeger energies and contractivity of dynamic transport distances in extended spaces.} 
Nonlinear Anal. \textbf{137} (2016), 77--134.

\bibitem[AH17]{AmbrosioHonda}
        \textsc{L. Ambrosio, S. Honda}:
        \textit{New stability results for sequences of metric measure spaces with uniform Ricci bounds from below.} 
Measure theory in non-smooth spaces, 1–51, Partial Differ. Equ. Meas. Theory, De Gruyter Open, Warsaw, 2017.

\bibitem[AMS16]{AmbrosioMondinoSavare16}
       \textsc{L. Ambrosio, A. Mondino, G. Savar\'e}:
       \textit{On the Bakry-\'Emery condition, the gradient estimates and the
Local-to-Global property of $RCD^*(K,N)$ metric measure spaces.}
       J. Geom. Anal. \textbf{26} (2016), 24--56.
       
\bibitem[ACDM15]{AmbrosioColomboDiMarino}
      \textsc{L. Ambrosio, M. Colombo, S. Di Marino}:
      \textit{Sobolev spaces in metric measure spaces: reflexivity and lower semicontinuity of slope.}
      Adv. Stud. in Pure Math. {\bf 67} (2015), 1--58.

\bibitem[AMS15]{AmbrosioMondinoSavare}
        \textsc{L. Ambrosio, A. Mondino, G. Savar\'e}:
\textit{Nonlinear diffusion equations and curvature conditions in metric measure spaces.}
ArXiv preprint 1509.07273, to appear in Mem. Amer. Math. Soc. 
	
	
\bibitem[AST16]{AmbrosioStraTrevisan}
        \textsc{L. Ambrosio, F. Stra, D. Trevisan}: 
        \textit{Weak and strong convergence of derivations and stability of flows with respect to MGH convergence.}
        J. Funct. Anal. \textbf{272} (2017), 1182–1229.

\bibitem[AMS16]{AmbrosioMondinoSavare2}
        \textsc{L. Ambrosio, A. Mondino, G. Savar\'e}:
         \textit{On the Bakry-\'Emery Condition, the Gradient Estimates and the Local-to-Global Property of $\RCD^*(K, N)$ Metric Measure Spaces.}
         J. Geom. Anal. \textbf{26} (2016), 24-56.
         
\bibitem[BE85]{BE}
          \textsc{D. Bakry, M. \'Emery}:
          \textit{Diffusions hypercontractives.}
          In S\'eminaire de Probabilit\'es, XIX, 1983/84. Lecture Notes in Math. 1123 177–206. Springer, Berlin.

\bibitem[BL06]{BL}
          \textsc{D. Bakry, M. Ledoux}:
          \textit{A logarithmic Sobolev form of the Li-Yau parabolic inequality.}
          Rev. Mat. Iberoam. \textbf{22} (2006), 683–702.

\bibitem[B17]{Bamler}
          \textsc{R. Bamler}:
          \textit{Structure theory of singular spaces.}
          J. Funct. Anal. \textbf{272} (2017), 2504--2627.

\bibitem[BKMR18]{BKMR}
	\textsc{J. Bertrand. C. Ketterer, I. Mondello, T. Richard}:
	\textit{Stratified spaces and synthetic Ricci curvature bounds}.
	ArXiv preprint 1804.08870.

\bibitem[BS18]{BS18}
          \textsc{E. Bru\`e, D. Semola}:
          \textit{Constancy of the dimension for $\RCD(K, N)$ spaces via regularity of Lagrangian flows.}
          ArXiv preprint, arXiv:1804.07128.

\bibitem[BBI01]{BuragoBuragoIvanov}
         \textsc{D. Burago, Y. Burago, S. Ivanov}: \textit{A course in metric geometry,} Graduate Studies in Mathematics, 33. American Mathematical Society, Providence, RI, 2001. 

\bibitem[CM16]{CavMil}
         \textsc{F. Cavalletti, E. Milman}: \textit{The Globalization Theorem for the Curvature Dimension Condition}. 
         ArXiv preprint 1612.07623.

\bibitem[Ch99]{Cheeger}
         \textsc{J. Cheeger}: \textit{Differentiability of Lipschitz functions on metric measure spaces}. 
         Geom. Funct. Anal. \textbf{9} (1999), 428--517.

\bibitem[CC96]{CheegerColding}
          \textsc{J. Cheeger, T. H. Colding}:
          \textit{Lower bounds on Ricci curvature and the almost rigidity of
warped products.}  Ann. of Math. \textbf{144} (1996), 189--237.

\bibitem[CC00]{CheegerColding3}
          \textsc{J. Cheeger, T. H. Colding}:
          \textit{On the structure of spaces with Ricci curvature bounded below, III.}  J. Differential Geom. \textbf{54} (2000), 37--74.

\bibitem[CW17]{ChenWang}
          \textsc{X. X. Chen, B. Wang}:
          \textit{Space of Ricci flows (II)-part A: moduli of singular Calabi-Yau spaces.}   Forum Math. Sigma \textbf{5} (2017), e32, 103 pp.

\bibitem[CN12]{CN}
          \textsc{T. H. Colding, A. Naber}:
          \textit{Sharp H\"older continuity of tangent cones for spaces with a lower Ricci curvature bound and applications.}  Ann. of Math. \textbf{176} (2012), 1173--1229.

\bibitem[DePhG17]{PhGi}
        \textsc{G. De Phillipis, N. Gigli}:
        \textit{Non-collapsed spaces with Ricci curvature bounded from below.}
         ArXiv preprint 1708.02060.

\bibitem[EKS15]{ErbarKuwadaSturm}
        \textsc{M. Erbar, K. Kuwada, K.-T. Sturm}:
        \textit{On the equivalence of the entropic curvature-dimension condition and Bochner's inequality on metric measure spaces.}
        Invent. Math. \textbf{201} (2015), 993--1071. 

\bibitem[G13]{Gigli0}
        \textsc{N. Gigli}:
        \textit{The splitting theorem in non-smooth context.}
      ArXiv preprint 1302.5555.

\bibitem[G15a]{Gigli1}
        \textsc{N. Gigli}:
        \textit{On the differential structure of metric measure spaces and applications.}
      Mem. Amer. Math. Soc. \textbf{236} (2015), no. 1113.  
	
\bibitem[G15b]{Gigli}
        \textsc{N. Gigli}:
        \textit{Nonsmooth differential geometry --
An approach tailored for spaces with Ricci curvature bounded from below.}
         Mem. Amer. Math. Soc. \textbf{251} (2018), no. 1196. 

\bibitem[HK00]{HK}
\textsc{P. Haj\l asz, P. Koskela},
\textit{Sobolev met Poincar\'e}. Mem. Amer. Math. Soc. \textbf{145} (2000), no. 688.


\bibitem[HP07]{HP}
        \textsc{C. Hinde, P. Petersen}:
        \textit{Generalized doubling meets Poincar\'e.}
       J. Geom. Anal. \textbf{7} (2007), 485–-494.
        
\bibitem[H14]{Honda4}
        \textsc{S. Honda}:
        \textit{A weakly second-order differential structure on rectifiable metric
measure space.} Geom. Topol. \textbf{18} (2014), 633-668.

\bibitem[H15]{Honda2}
        \textsc{S. Honda}:
        \textit{Ricci curvature and $L^p$-convergence.}
        J. Reine Angew Math. \textbf{705} (2015), 85--154.

\bibitem[J14]{Jiang}
	\textsc{R. Jiang}:
	\textit{Cheeger-harmonic functions in metric measure spaces revisited}.
	J. Funct. Anal. {\bf 266}  (2014),  1373--1394.

\bibitem[KR15]{KR}
	\textsc{C. Ketterer, T. Rajala}:
	\textit{Failure of topological rigidity results for the measure contraction property}.
	Potential Anal., \textbf{42} (2015), 645--655.

\bibitem[KM96]{JM}
	\textsc{J. Kinnunen, O. Martio}:
	\textit{The Sobolev capacity on metric spaces}.
	Ann. Acad. Sci. Fenn. Math. \textbf{21} (1996), 367–-382.

\bibitem[LV09]{LottVillani}
\textsc{J. Lott, C. Villani}:
\textit{Ricci curvature for metric-measure spaces via optimal transport}. 
Ann. of Math. \textbf{169} (2009), 903--991.
        
\bibitem[Raj12]{Rajala}
        \textsc{T. Rajala}:
        \textit{Local Poincar\'e inequalities from stable curvature conditions on metric spaces}.
        Calc. Var. Partial Differential Equations \textbf{44}(3) (2012), 477--494.
       
        
\bibitem[St06]{Sturm06}
	\textsc{K.-T. Sturm}:
	\textit{On the geometry of metric measure spaces, I and II}.
	Acta Math. \textbf{196} (2006), 65--131 and 133--177.

	
\end{thebibliography}
\end{document}